\documentclass[11pt]{amsart}
\usepackage{fullpage}
\usepackage{mathtools}
\usepackage{amsmath}
\usepackage{amsthm}
\usepackage{amssymb}
\usepackage{color}
\usepackage{hyperref}

\newtheorem{theorem}{Theorem}[section]

\newtheorem{lemma}[theorem]{Lemma}

\newtheorem{remark}[theorem]{Remark}

\newtheorem{problem}[theorem]{Problem}

\def\cC{{\mathcal{C}}}
\def\cF{{\mathcal{F}}}

\title{Monochromatic triangles with empty intersection \\ 
and Kneser Ramsey numbers }
\author{Igor Araujo}
\address{Department of Mathematical Sciences, University of Memphis, Memphis, TN 38111} 
\email{\parbox[t]{\linewidth}{iaaraujo@memphis.edu}.} 
\date{}

\begin{document}

\begin{abstract}
    Recently, Heath, McCourt, Parker, Schwieder, and Zerbib~\cite{heath} initiated the systematic study of the \emph{$r$-Kneser Ramsey number} $R_r^{KG}(s,t)$ and investigated related Ramsey-type problems.
    A central motivation for their work comes from a question of Holmsen, Hrusak, and Rold\'an-Pensado, who asked whether, for $n=2k-1$ and sufficiently large $k$, every red/blue edge-coloring of the complete graph on the vertex set $V = \binom{[n]}{k}$ necessarily contains a monochromatic triangle $ABC$ with $A,B,C \in V$ and $A \cap B \cap C = \emptyset$. 
    Heath, McCourt, Parker, Schwieder, and Zerbib established that this conclusion holds when $n \ge \frac{7k}{3}$ and $k\ge 12$. 
    We make substantial progress toward the problem of Holmsen, Hrusak, and Rold\'an-Pensado by proving that the conclusion already holds for every $k\ge 2$ whenever $n\ge 2k+1$. In addition, we obtain improved lower bounds for $R_r^{KG}(s,t)$  when $s$ and $t$ are fixed and $r$ is sufficiently large. 
\end{abstract}

\maketitle

\section{Introduction}

The classical Ramsey numbers $R(s,t)$ are among the most important parameters in extremal combinatorics~\cite{graham, ramsey, verstraete}, measuring the smallest integer $n$ for which every red/blue coloring of the edges of the complete graph $K_n$ necessarily contains a red copy of $K_s$ or a blue copy of $K_t$.
The Kneser graph $KG(n,r)$ is another fundamental object in combinatorics and topological graph theory~\cite{alon, kneser, lovasz}; its vertices are the $r$-subsets of $[n]$, and two vertices are adjacent precisely when the corresponding sets are disjoint.
Merging these two themes, Heath, McCourt, Parker, Schwieder, and Zerbib~\cite{heath} introduced the \emph{$r$-Kneser Ramsey number} $R_r^{KG}(s,t)$ as the minimum integer $n$ such that every red/blue edge-coloring of $KG(n,r)$ contains a red $K_s$ or a blue $K_t$. One of their motivations for studying these numbers comes from the following problem proposed by Holmsen, Hrusak, and Rold\'an-Pensado. 

\begin{problem}[Problem 3 in~\cite{heath}] \label{prob}
    Is it true that, for large enough $k$, in every red/blue edge-coloring of the complete graph on the vertex set $V = \binom{[2k-1]}{k}$ there is a monochromatic triangle $ABC$ with $A,B,C \in V$ and $A \cap B \cap C = \emptyset$.
\end{problem}

In order to avoid repeating the condition $A \cap B \cap C = \emptyset$ multiple times in the paper, we say that the triangle $ABC$ is \emph{non-intersecting} if $A \cap B \cap C = \emptyset$.
Heath, McCourt, Parker, Schwieder, and Zerbib made progress on Problem~\ref{prob} through the following result.

\begin{theorem}[Heath, McCourt, Parker, Schwieder, Zerbib~\cite{heath}] \label{thm:heath}
    Let $k \ge 12$ be an integer divisible by 6. Then in every red/blue edge-coloring of the complete graph on the vertex set $\binom{[\frac{7k}{3}]}{k}$ there is a monochromatic non-intersecting triangle.
\end{theorem}

The main result of this note provides a substantial improvement to Theorem~\ref{thm:heath}.

\begin{theorem}\label{thm:main}
    Let $k \ge 2$. Then, in every red/blue edge-coloring of the complete graph on the vertex set $V = \binom{[2k+1]}{k}$, there is a monochromatic non-intersecting triangle.
\end{theorem}

We also obtain improved lower bounds for the Kneser Ramsey numbers.

\begin{theorem}\label{thm:kneser}
    For every $r\ge 1$ and $s\ge t\ge 2$, we have that $R_r^{KG}(s,t) \ge sr + (t-2)(s-1)$.
\end{theorem}

\begin{remark}
    For fixed $s,t\ge 3$, the bound in Theorem~\ref{thm:kneser} improves upon the estimate $R_r^{KG}(s,t) \ge R(s,t) + 2r-2 $ from Proposition 10 of~\cite{heath} provided that $r$ is sufficiently large.
    In fact, for the case $t=3 < s \le 6$, our result already yields stronger bounds even when $r=3$:
$$ R_3^{KG}(4,3) \ge 15 \, ,  \quad  R_3^{KG}(5,3) \ge 19 \, , \quad \text{and} \quad R_3^{KG}(6,3) \ge 23 .$$
\end{remark}

\subsection*{Organization of the paper} In section~\ref{sec:main}, we prove Theorem~\ref{thm:main}, and in section~\ref{sec:kneser}, we prove Theorem~\ref{thm:kneser}.

\section{Monochromatic triangles in $\binom{[2k+1]}{k}$} \label{sec:main}

In this section, we will prove Theorem~\ref{thm:main}. We first start with some observations. It is a known fact that every red/blue edge-coloring of $K_6$ contains two monochromatic triangles (see, e.g., Fact 3.3 in~\cite{balogh}). We need the following slightly stronger statement, where we obtain vertex-disjoint monochromatic triangles of the same color while avoiding monochromatic non-intersecting triangles.

\begin{lemma}\label{lem:6sets}
    Assume $A_1,A_2,A_3, B_1, B_2, B_3 \in V = \binom{[n]}{k}$ are sets such that for every $r,s,t \in [3]$ where $r\neq s$, we have that $A_tB_rB_s$ and $A_r A_s  B_t $ are non-intersecting. Then, in every red/blue edge-coloring of the complete graph on the vertex set $V$, either there is a monochromatic non-intersecting triangle, or $A_1A_2A_3$ and $B_1B_2B_3$ are both monochromatic triangles of the same color.
\end{lemma}
    
\begin{proof}
    We consider the coloring induced on the vertices $A_1,A_2,A_3, B_1, B_2, B_3 \in \binom{[n]}{k}$. There are at least two monochromatic triangles; either at least one of these triangles is non-intersecting, or they are precisely $A_1A_2A_3$ and $B_1B_2B_3$. Suppose that $A_1A_2A_3$ is colored red and $B_1B_2B_3$ is colored blue.
    For any $r\in [3]$, since $A_r B_s B_t$ is non-intersecting, if both $A_rB_s$ and $A_rB_t$ are colored blue, then there is a monochromatic non-intersecting triangle and we are done. Otherwise, each $A_r$ is incident to at most one blue edge into $\{B_1, B_2, B_3\}$. Consequently, there are at most three blue edges between $\{A_1, A_2, A_3\}$ and $\{B_1, B_2, B_3\}$. Analogously, there are at most three red edges between $\{A_1, A_2, A_3\}$ and $\{B_1, B_2, B_3\}$, which is a contradiction.
\end{proof}

The following lemma ensures that any bipartition of the family $W = \binom{[2t]}{t-1}$ inevitably yields a pair of sets on one side that is jointly disjoint from one of the two canonical halves of $[2t]$. This structural obstruction will be crucial later when we apply it to find a monochromatic non-intersecting triangle.

\begin{lemma}\label{lem:partition}
    Let $t \ge 2$ and $W = \binom{[2t]}{t-1}$. For every partition $W = W_1 \cup W_2$ there is either $S,T \in W_1$ such that $S\cap T \cap \{1, \dots, t\} = \emptyset$ or $S,T \in W_2$ such that $S\cap T \cap \{t+1, \dots, 2t\} = \emptyset$. 
\end{lemma}

\begin{proof}
    We split the proof into two cases according to the parity of $t$.

    \textbf{Case 1:} $t=2s$ is even and $W = \binom{[4s]}{2s-1}$.
    Set 
    $$ A = \{1, \dots, s \}  \, , \, 
    B = \{s + 1, \dots, 2s \} \, , \, 
    C = \{2s+1, \dots, 3s-1 \} \, , \text{ and } \, 
    D = \{3s +1, \dots, 4s -1 \} .$$ 
    Note that $|A|=|B|= s$ and $|C|=|D| = s-1$. 

    \textbf{Case 2:} $t=2s+1$ is odd and $W = \binom{[4s+2]}{2s}$. 
    Set 
    $$ A = \{1, \dots, s \} \, , \,
    B = \{s + 1, \dots, 2s \} \, , \,
    C = \{2s+2, \dots, 3s+1 \} \, , \text{ and } \,
    D = \{3s +2, \dots, 4s+1 \} .$$ 
    Note that $|A|=|B|=|C|=|D| = s$. 

    In either case, we analyze the partition induced by the sets 
    $$ A\cup C,\, A \cup D,\, B\cup C,\, B \cup D \in W .$$ 
    If both $A\cup C$ and $A\cup D$ lie in $W_2$, then we are done, since 
    $$ C\cap D = \emptyset \quad \text{ and } \quad
    A \cap \{t+1\, ,\, \dots \, , \, 2t\} = \emptyset.$$ 
    Thus, assume that $A \cup C \in W_1$. If either $B\cup C$ or $B \cup D$ lies in $W_1$, then again we are done, because 
    $$ A \cap B = \emptyset \, ,\, \quad 
    C \cap \{1\, ,\, \dots \, , \, t\} = \emptyset \, ,\, \quad \text{ and } \quad
    D \cap \{1\, ,\, \dots \, , \, t\} = \emptyset .$$
    Finally, if both $B \cup C$ and $B \cup D$ lie in $W_2$, then we are also done,  as 
    $$ C\cap D = \emptyset \quad \text{ and } \quad
    B \cap \{t+1\, ,\, \dots \, , \, 2t\} = \emptyset. \eqno\qedhere $$ 
\end{proof}

We are now ready to prove our main result.

\begin{proof}[Proof of Theorem~\ref{thm:main}]
Let $n = 2k+1$. For every $S \subseteq [n]$ with $|S| = k-1$ and every $i \in [n] \setminus S$, define $S^{i} \coloneq S \cup \{i\}$. 
We first show that, for every $X \subseteq [n]$ with $|X| = k-1$, the family $\{X^{i} : i \in [n] \setminus X\}$ induces a monochromatic clique. 

Indeed, for arbitrary $r,s,t \notin X$, choose a set $Y \subseteq [n] \setminus (X \cup \{r,s,t\})$ of size $|Y| = k-1$. Such a set exists because $n \ge 2(k-1)+3 = 2k+1$. Since the collection
\[
\{X^{r}, X^{s}, X^{t}, Y^{r}, Y^{s}, Y^{t}\}
\]
satisfies the hypothesis of Lemma~\ref{lem:6sets} with $\{X^r, X^s, X^t\}$ and $\{Y^r, Y^s, Y^t\}$ playing the roles of $\{A_1, A_2, A_3\}$ and $\{B_1, B_2, B_3\}$, respectively, we conclude that $\{X^{r}, X^{s}, X^{t}\}$ is monochromatic, otherwise there is a monochromatic non-intersecting triangle. As $r,s,t \notin X$ were arbitrary, the entire $X$-clique $\{X^{i} : i \in [n] \setminus X\}$ is monochromatic.

Moreover, Lemma~\ref{lem:6sets} implies that if $|X| = |Y| = k-1$ and $X \cap Y = \emptyset$, then the $X$-clique and the $Y$-clique are monochromatic of the same color. Since the Kneser graph $KG(n,k-1)$ is connected, it follows that all $S$-cliques with $|S| = k-1$ have the same color; let us say they are all blue.

Let $\ell$ be minimal such that there exist sets $A,B \in V$ with $|A \cap B| = k - \ell$ and $AB$ is colored red. The discussion above shows that $\ell \ge 2$. Furthermore, since no triple $A B C$ with $A \cap B \cap C = \emptyset$ can be monochromatic, we also have $\ell \le \lceil k/2 \rceil$.

Without loss of generality, assume $AB$ is colored red, where
\[
A = \{1, \dots, \ell,\, 2\ell+1, \dots, k+\ell\}
\quad\text{and}\quad
B = \{\ell+1, \dots, 2\ell,\, 2\ell+1, \dots, k+\ell\}.
\]
Set $F = \{k+\ell+1, \dots, 2k+1\}$. Then $|F| = k - \ell + 1$, so any edge $XY$ with $F \subseteq X \cap Y$ is necessarily colored blue.
Define
\[
\cC \coloneq \{\, F \cup X : X \subseteq [2\ell],\ |X| = \ell - 1 \,\}.
\]
Then $\cC$ is monochromatic blue. For each $C \in \cC$, note that $AC$ and $BC$ cannot be simultaneously red. Partition $\cC$ as
\[
\cC_A = \{C \in \cC : AC \text{ is blue}\} \, ,
\qquad
\cC_B = \{C \in \cC : AC \text{ is red and } BC \text{ is blue}\}.
\]
By Lemma~\ref{lem:partition}, there exist $C_1, C_2 \in \cC$ such that either
\begin{itemize}
    \item $A \cap C_1 \cap C_2 = \emptyset$ and both $AC_1$ and $AC_2$ are blue, or
    \item $B \cap C_1 \cap C_2 = \emptyset$ and both $BC_1$ and $BC_2$ are blue. \hfill \qedhere
\end{itemize}
\end{proof}

\section{An edge-coloring avoiding monochromatic cliques} \label{sec:kneser}

In this section, we present a coloring that proves Theorem~\ref{thm:kneser}.
The coloring is inspired by the extremal families from~\cite{alon} and is a direct extension of the coloring used in the proof of Theorem~11 in~\cite{heath}.  

\begin{proof}[Proof of Theorem~\ref{thm:kneser}]
Let $N = (sr - 1) + (t-2)(s-1)$.
Consider the partition $\cF_1 \cup \dots \cup \cF_{t-1} = \binom{[N]}{r}$ defined as follows.  
Set 
\[
\cF_1 \coloneq \binom{[sr - 1]}{r},
\]
and for each $i \ge 2$, define
\[
\cF_i \coloneq \left\{ X \in \binom{[N]}{r} \setminus \bigcup_{j<i} \cF_j : 
X \cap \{ sr + (i-2)(s-1)\, , \, \dots \, , \, sr + (i-2)(s-1) + s-2 \} \neq \emptyset \right\}. 
\]

Observe that for every $i \in [t-1]$, the family $\cF_i$ does not contain a copy of $K_s$.  
We color every edge within each $\cF_i$ red, and every edge joining vertices in different $\cF_i$ blue. Under this coloring, there is no red $K_s$ and no blue $K_t$. Indeed, the graph induced by color blue is $(t-1)$-partite and each component of the graph induced by color red is within some part $\cF_i$, which is $K_s$-free. Consequently,
$$ R_r^{KG}(s,t) \;\ge\; N+1 = sr + (t-2)(s-1). \eqno\qedhere $$
\end{proof}

\section*{Acknowledgements}

I am grateful to J\'ozsef Balogh and the anonymous referees for useful comments on a previous version of this manuscript.

\end{document}